\documentclass{amsproc}

\usepackage{amssymb}

\usepackage{accents}
\usepackage[english,german]{babel}
\usepackage{cite}

\newtheorem{theorem}{Theorem}[section]
\newtheorem{lemma}[theorem]{Lemma}
\newtheorem{Prop}[theorem]{Proposition}

\theoremstyle{definition}
\newtheorem{definition}[theorem]{Definition}

\theoremstyle{remark}
\newtheorem{remark}[theorem]{Remark}

\numberwithin{equation}{section}

\makeatletter
\newcommand*{\rom}[1]{\expandafter\@slowromancap\romannumeral #1@}
\makeatother
\newcommand{\R}{\mathbb{R}}

\newcommand{\Z}{\mathbb{Z}}

\newcommand{\Ric}{\operatorname{Ric}}
\newcommand{\Rm}{\operatorname{Rm}}
\newcommand{\Scal}{\operatorname{R}}
\newcommand{\tr}{\operatorname{tr}}
\newcommand{\grad}{\operatorname{grad}}

\newcommand{\signum}{\operatorname{sign}}
\newcommand{\schild}{Schwarz\-schild }

\newcommand{\mylap}[1]{{}^{#1}\!\triangle}
\newcommand{\my}[2]{{}^{#1}{#2}}

\newcommand{\free}[1]{\accentset{\,\circ}{#1}}

\newcommand{\photo}{P^{3}}
\newcommand{\Sphoto}{\overline{P}^{3}}  
\newcommand{\slice}{M^{3}}

\newcommand{\surf}{\Sigma^{2}}

\begin{document}
\selectlanguage{english}
\title[Uniqueness of photon spheres in static vacuum as.\ flat spacetimes]{Uniqueness of photon spheres in static vacuum asymptotically flat spacetimes}

\author{Carla Cederbaum}
\address{Mathematics Department, Universit\"at T\"ubingen, Germany}
\email{cederbaum@math.uni-tuebingen.de}
\thanks{The author was supported by the Robert Bosch Foundation. This material is based upon work supported by the National Science Foundation under Grant No.\,0932078 000, while the author was in residence at the Mathematical Sciences Research Institute in Berkeley, California, during the fall semester 2013.}

\subjclass[2010]{Primary 35Q75, Secondary 83C15, 83C20, 35H99, 53A99, 53Z05}

\date{}
\keywords{General relativity, null geodesics, photon sphere, static spacetimes}

\begin{abstract}
Adapting Israel's proof of static black hole uniqueness \cite{Israel}, we show that the Schwarzschild spacetime is the only static vacuum asymptotically flat spacetime that possesses a suitably defined photon sphere. 
\end{abstract}

\maketitle

\section{Introduction}\label{sec:intro}
The static spherically symmetric \schild black hole spacetime\footnote{The same formula still defines a \emph{\schild spacetime} if $m<0$. The corresponding metric is well-defined on $\overline{\mathfrak{L}}^{4}=\R\times(\R^{3}\setminus\lbrace0\rbrace)$ but possesses neither a black hole horizon nor a photon sphere. If $m=0$, the Schwarzschild spacetime degenerates to the Minkowski spacetime.} of mass $m>0$ can be represented as $(\overline{\mathfrak{L}}^{4}:=\R\times(\R^{3}\setminus B_{2m}(0)),\overline{\mathfrak{g}})$, where the Lorentzian metric $\overline{\mathfrak{g}}$ is given by
\begin{align}\label{schwarzmetric}
\overline{\mathfrak{g}}&=-\overline{N}^{2}dt^{2}+\overline{N}^{-2}dr^{2}+r^{2}\Omega,\quad 
\overline{N}=\left(1-\frac{2m}{r}\right)^{1/2},
\end{align}
with $\Omega$ denoting the canonical metric on $\mathbb{S}^{2}$. In these coordinates, the black hole horizon is given by the cylinder $\R\times\mathbb{S}^{2}_{2m}=\lbrace r=2m\rbrace$. The timelike submanifold $\Sphoto:=\R\times\mathbb{S}^{2}_{3m}=\lbrace r=3m\rbrace$ is called a \emph{photon sphere} because any null geodesic of $(\overline{\mathfrak{L}}^{4},\overline{\mathfrak{g}})$ that is initially tangent to $\Sphoto$ remains tangent to it. The \schild photon sphere thus models (an embedded submanifold ruled by) photons spiraling around the central black hole ``at a fixed distance''. 

Apart from its phe\-no\-me\-nological significance for general relativity, the \schild photon sphere and its generalized analog in the Kerr spacetime are crucially relevant for questions of dynamical stability in the context of the Einstein equations, see e.\,g.\ \cite{DR}. It thus seems useful to understand photon spheres in more generality\footnote{Clearly, from the perspective of stability, it will be necessary to understand the existence of (generalized) photon spheres in dynamical or at least stationary, not only in static spacetimes.}.

Photon spheres have also been studied in the context of gravitational lensing. There, they are related to the existence of relativistic images as was demonstrated by Virbhadra and Ellis \cite{VE1,VE2} in the context of static spherically symmetric spacetimes. Building upon this work, Claudel, Virbhadra and Ellis \cite{CVE} gave a geometric definition of photon spheres, again for static spherically symmetric spacetimes.

To the best knowledge of the author, it is unknown whether more general spacetimes can possess (generalized) photon spheres, see p.\;838 of \cite{CVE}. We will address this question for asymptotically flat static vacuum or \emph{AF-geometrostatic} spacetimes.

In Section \ref{sec:setup and definition}, we will give a geometric definition of photon spheres in AF-geometrostatic spacetimes. We will explain how our definition generalizes the one given in \cite{CVE}. Our definition of photon spheres is related to constancy of the energy of the null geodesics generating the photon sphere (as observed by the static observers in the spacetime), see Lemma \ref{lem:energy}. In Section \ref{sec:proof}, we will prove that the only AF-geometrostatic spacetime admitting a photon sphere is the Schwarzschild spacetime:
\begin{theorem}
Let $(\mathfrak{L}^4,\mathfrak{g})$ be an AF-geometrostatic spacetime possessing a photon sphere $\photo\hookrightarrow\mathfrak{L}^4$  with mean curvature $\mathfrak{H}$, arising as the inner boundary of $\mathfrak{L}^{4}$. Assume that the lapse function $N$ regularly foliates $\mathfrak{L}^4$. Then $\mathfrak{H}\equiv\text{const}$ and $(\mathfrak{L}^4,\mathfrak{g})$ is isometric to the \schild spacetime of the same  mass $m=1/(\sqrt{3}\,\mathfrak{H})>0$.
\end{theorem}

For this, we adapt Israel's proof of uniqueness of black holes in asymptotically flat static vacuum spacetimes \cite{Israel} (as exposed in Heusler \cite{Heusler}). As Israel, we will have to assume that the lapse function regularly foliates the spacetime (at least in the region exterior to the photon sphere). This is automatically true in a neighborhood of the asymptotic (spacelike) infinity if the ADM-mass of the spacetime is non-zero, s.\ Lemma \ref{lem:foli}. Moreover, the assumption holds true in the \schild spacetime as well as in most known AF-geometrostatic spacetimes. In particular, this assumption restricts our attention to \emph{connected} photon spheres that are indeed \emph{topological spheres}. 

Israel's result has been generalized in other directions, for example by Bunting and Masood-ul-Alam \cite{BMuA} and by Miao \cite{MiaoUnique}. They generalized Israel's static vacuum black hole uniqueness theorem, removing the technical condition of the lapse function foliating the spacetime outside the horizon, and thus in particular allowing non-spherical and disconnected  horizons a priori. Following the Bunting and Masood-ul-Alam approach, a priori disconnected and not necessarily spherical photon ``spheres'' will be addressed in \cite{cedergal}, together with other results on photon surfaces.

The author would like to thank Gregory Galloway and Gerhard Huisken for helpful discussions.
\vfill

\section{Setup and definitions}\label{sec:setup and definition}
Let us first quickly review the definition of and some facts about asymptotically flat static vacuum spacetimes. These model exterior regions of static configurations of stars or black holes. See Bartnik \cite{Bartnik} for a more detailed account of asymptotically flat Riemannian manifolds and harmonic coordinates as well as for the definition of the weighted Sobolev spaces $W^{k,p}_{-\tau}(E)$ we will use in the following. More details and facts on asymptotically flat static vacuum spacetimes can be found in \cite{CDiss}.

\begin{definition}[AF-geometrostatic spacetimes and systems]\label{def:AFSVS}
A smooth Lorent\-zian manifold or \emph{spacetime} $(\mathfrak{L}^{4},\mathfrak{g})$ is called \emph{(standard) static} if there is a smooth Riemannian manifold $(\slice,g)$ and a smooth \emph{lapse} function $N:\slice\to\R^{+}$ s.\,t.\
\begin{align}\label{static}
\mathfrak{L}^{4}&=\R\times \slice,\quad \mathfrak{g}=-N^{2}dt^{2}+g,
\end{align}
and \emph{vacuum} if it satisfies the Einstein vacuum equation
\begin{align}\label{EE}
\mathfrak{Ric}&=0,
\end{align}
where $\mathfrak{Ric}$ denotes the Ricci curvature tensor of $(\mathfrak{L}^{4},\mathfrak{g})$. We will sometimes call $\slice$ a \emph{(time-)slice} of $\mathfrak{L}^{4}$, as it arises as $\slice=\lbrace t=0\rbrace$, where $t$ is considered as the time variable of the spacetime. A static spacetime is called \emph{asymptotically flat} if the manifold $\slice$ is diffeomorphic to the union of a (possibly empty) compact set and an open \emph{end} $E^{3}$ which is diffeomorphic to $\R^{3}\setminus \overline{B}$, where $B$ is the open unit ball in $\R^{3}$. Furthermore, we require that the lapse function $N$, the metric $g$, and the coordinate diffeomorphism $\Phi=(x^{i}):E^{3}\to\R^{3}\setminus \overline{B}$ combine s.\,t.\
\begin{align}\label{AF}
g_{ij}-\delta_{ij}&\in W^{k,2}_{-\tau}(E)\\\label{NAF}
N-1\;\;&\in W^{k+1,2}_{-\tau}(E)
\end{align}
for some $\tau>1/2$, $\tau\notin\Z$, $k\geq3$ and that $\Phi_*g$ is uniformly positive definite and uniformly continuous on $\R^{3}\setminus \overline{B}$. Here, $\delta$ denotes the Euclidean metric on $\R^{3}$. For brevity, smooth\footnote{M\"uller zum Hagen \cite{MzH} showed that static spacetimes with $g_{ij}, N\in C^3$ are automatically real analytic with respect to wave-harmonic coordinates if they solve \eqref{EE}, see also Footnote 5.} asymptotically flat maximally extended\footnote{i.\,e.\ geodesically complete up to a possible inner boundary} static vacuum spacetimes will be referred to as \emph{AF-geometrostatic spacetimes}, the associated triples $(\slice,g,N)$ will be called \emph{AF-geometrostatic systems}. We will frequently use the radial coordinate $r:=\sqrt{(x^{1})^{2}+(x^{2})^{2}+(x^{3})^{2}}$ corresponding to the chosen coordinates $(x^{i})$.
\end{definition}

Exploiting \eqref{static}, the Einstein vacuum equation \eqref{EE} reduces to\begin{align}\label{SMEvac1}
N\,{\Ric}&={\nabla}^2 N\\\label{SMEvac2}
\Scal&=0
\end{align}
on $\slice$, where $\nabla^{2}$, $\Ric$, and $\Scal$ denote the covariant Hessian, the Ricci, and the scalar curvature of the metric $g$, respectively. Combining \eqref{SMEvac1} and \eqref{SMEvac2}, one obtains
\begin{align}\label{SMEvac3}
\triangle N&=0
\end{align}
on $\slice$, where $\triangle$ denotes the Laplacian with respect to $g$. The \emph{static metric equations} \eqref{SMEvac1}, \eqref{SMEvac3} are a system of degenerate elliptic quasi-linear second order PDEs in the variables $N$ and $g_{ij}$ (with respect to for example $g$-harmonic coordinates). Translating a result by Kennefick and \'O Murchadha \cite{KM} to our notation, we find that AF-geometrostatic systems are automatically \emph{asymptotically Schwarzschildean}: 

\begin{theorem}[Kennefick \& \'O Murchadha]\label{thm:KM}
Let $(\slice,g,N)$ be an AF-geo\-metro\-static system as in Definition \ref{def:AFSVS} with an end $E^{3}$, $\tau>1/2$, $\tau\notin\Z$, and $k\geq3$. Then $(\slice,g,N)$ must be \emph{asymptotically Schwarzschildean}, i.\,e.\ be such that $g, N$ satisfy
\begin{align}\label{eq:KM1}
g_{ij}-\overline{g}_{ij}&\in W^{k,2}_{-(\tau+1)}(E)\\\label{eq:KM2}
N-\overline{N}\;\,&\in W^{k+1,2}_{-(\tau+1)}(E)
\end{align}
with respect to $\mathfrak{g}$-wave harmonic coordinates\footnote{In \cite{KM}, the condition on the coordinates is that they should be $\gamma$-harmonic with respect to the conformally transformed metric $\gamma:=N^2 g$. This is equivalent to them being $\mathfrak{g}$-wave harmonic, see e.\,g.\ Lemma 3.1.3 in \cite{CDiss}.} with respect to the associated spacetime metric $\mathfrak{g}$ defined by \eqref{static}. Here, $\overline{N}$ and $\overline{g}$ are the lapse function and the Riemannian metric corresponding to the \schild metric\footnote{or $N\equiv1$ and the Euclidean metric if $m=0$.} \eqref{schwarzmetric} of some mass parameter $m\in\R$.
\end{theorem}

\begin{remark}
A simple computation shows that the parameter $m$ in Theorem \ref{thm:KM} equals the ADM-mass of the spacetime for the definition of which we refer the reader to Arnowitt, Deser, and Misner \cite{ADM}.
\end{remark}
In Theorem \ref{thm:main}, we assume that the AF-geometrostatic spacetimes under consideration are foliated by the level sets of the lapse function $N$ (outside their respective photon spheres). This restricts the topology of $\slice$ (outside the photon sphere) to that of $\R^3\setminus \overline{B}$ and the topology of the level sets of $N$ in $(\mathfrak{L}^{4},\mathfrak{g})$ to $\R\times\mathbb{S}^{2}$. The Schwarzschild spacetime is clearly foliated in this way. Moreover, any AF-geometrostatic spacetime is foliated like this in a neighborhood of spatial infinity:

\begin{lemma}\label{lem:foli}
Let $(\slice,g,N)$ be an AF-geometrostatic system with non-vanishing ADM-mass $m$. Then there exists a compact interior $K\subset \slice$ such that $N$ foliates the slice $\slice\setminus K$ with spherical leaves and $\R\times(\slice\setminus K)\subset\mathfrak{L}^{4}$ with cylindrical leaves $\R\times\mathbb{S}^{2}$. 
\end{lemma}
\begin{proof}
Theorem \ref{thm:KM} tells us that $$\frac{\partial N}{\partial x^i}-\frac{mx_i}{r^3}\in W^{k,2}_{-\tau-2}(E)$$ holds in $\mathfrak{g}$-wave harmonic asymptotically flat coordinates in the end $E$ of $\slice$. Therefore, $dN\neq0$ holds in a neighborhood of infinity. By the implicit function theorem, $N$ thus locally foliates $\slice\setminus K$ for a suitably large compact interior $K$. The leaves of the foliation must be spherical as $N=1-m/r+\mathcal{O}(r^{-2})$ is radial up to second order again by Theorem \ref{thm:KM}.
\end{proof}
\vfill
\subsection{Definition of photon surfaces and photon spheres}
Let us now proceed to defining photon spheres in AF-geometrostatic spacetimes. First, let us quickly review the history of the definition of photon spheres in static spacetimes\footnote{In stationary non-static spacetimes, one cannot expect photon spheres to arise as embedded submanifolds of the spacetime as the (angular) momentum of the spacetime will affect photons orbiting in one way differently from those orbiting the other way, as is well-known for the Kerr spacetime, see e.\,g.\ O'Neill \cite{ONeill2}.}. In the context of static spherically symmetric spacetimes\footnote{not necessarily subject to Einstein's equation}, Virbhadra and Ellis \cite{VE1,VE2} defined photon spheres to be timelike hypersurfaces of the form $\lbrace r=r_{0}\rbrace$, where ``$r_{0}$ is the closest distance of approach for which the Einstein bending angle of a light ray is unboundedly large'' (cited from \cite{CVE}).

Claudel, Virbhadra and Ellis \cite{CVE} have geometrized this definition of a photon sphere, again in the context of static spherically symmetric spacetimes. They first define a \emph{photon surface} to be an immersed no-where spacelike hypersurface $P^{n}$ of a general Lorentzian spacetime $(\mathfrak{L}^{n+1},\mathfrak{g})$ such that every tangent vector $X\in TP^{n}$ can be extended to a null geodesic $\gamma:(-\varepsilon,\varepsilon)\to\mathfrak{L}^{n+1}$ remaining within $P^{n}$ and satisfying $\dot{\gamma}(0)=X$. A \emph{photon sphere} in a static spherically symmetric spacetime $(\mathfrak{L}^{n+1},\mathfrak{g})$ is then defined as an $\R\times SO(n)$-invariant photon surface $P^{n}\hookrightarrow\mathfrak{L}^{n+1}$.

To the best knowledge of the author, it is an open question whether more general spacetimes can possess photon spheres in this or in a generalized sense, see p.\;838 of \cite{CVE}. We will address this question in the context of AF-geometrostatic spacetimes $(\mathfrak{L}^{4},\mathfrak{g})$ as defined in Definition \ref{def:AFSVS}.

Specializing the definition of Claudel, Virbhadra and Ellis \cite{CVE}, we make the following definition of \emph{photon surfaces}, see also Perlick \cite{Perlick}.

\begin{definition}[Photon surface]\label{def:photo-surf}
A timelike embedded hypersurface $\photo\hookrightarrow\mathfrak{L}^{4}$ of an AF-geometrostatic spacetime $(\mathfrak{L}^{4},\mathfrak{g})$ is called a \emph{photon surface} if and only if any null geodesic initially tangent to $\photo$ remains tangent to $\photo$ as long as it exists.
\end{definition}

The \schild photon sphere clearly is a photon surface in the \schild spacetime. Moreover, by spherical symmetry and strict monotonicity of $\overline{N}$, a hypersurface of the form $\lbrace r=r_{0}\rbrace$ can also be written as $\lbrace \overline{N}=\overline{N}_{0}:=\overline{N}(r_{0})\rbrace$ in the \schild spacetime. The same is true in all static spherically symmetric spacetimes and thus in the situation considered in \cite{CVE} as long as $N$ is strictly monotone. We may thus consistently replace level sets of the radial variable $r$ related to spherical symmetry by level sets of the lapse function $N$ in a general AF-geometrostatic spacetime. This allows us to make the following definition of \emph{photon spheres} in AF-geometrostatic spacetimes.

\begin{definition}[Photon sphere]\label{def:photo}
Let $(\mathfrak{L}^{4},\mathfrak{g})$ be an AF-geometrostatic spacetime, $\photo\hookrightarrow\mathfrak{L}^{4}$ a photon surface. Then $\photo$ is called a \emph{photon sphere} if the lapse function $N$ of the spacetime is constant along $\photo$ or in other words if $\photo=\lbrace N=N_{0}\rbrace$.
\end{definition}

This clearly generalizes the definition of photon spheres given in \cite{CVE}, thus making the \schild photon sphere a photon sphere in our sense in particular. Moreover, our definition extends certain physical properties of the \schild photon sphere; the condition that the lapse function $N$ be constant along the photon sphere hence is not merely a technical extension of the spherically symmetric case. It has in fact a very immediate physical interpretation: The energy $E$ and the associated frequency $\nu=E/\hbar$ of a null geodesic (photon) $\gamma$ observed by the static observers ${N}^{-1}\partial_{t}$ is constant if and only if $N$ is constant along the geodesic $\gamma$, see Lemma \ref{lem:energy} below. Hence \emph{all} null geodesics tangent to a photon surface $\photo$ have constant energy in the eyes of the static observers $\frac{1}{N}\partial_{t}$ if and only if the lapse function is constant along the photon surface.

This constant energy is a main reason why the photon sphere in the \schild spacetime makes the analysis of dynamical stability difficult: The energy of photons and thus also of waves traveling with speed of light does not disperse along the photon sphere. It thus seems justified to generalize the \schild photon sphere and the notion of photon sphere defined in \cite{CVE} by defining photon spheres $\photo$ as photon surfaces satisfying $\photo=\lbrace N=N_{0}\rbrace$.

\begin{lemma}[Constant energy]\label{lem:energy}
Let $(\mathfrak{L}^{4},\mathfrak{g})$ be an AF-geometrostatic spacetime and $\gamma$ a null geodesic in $(\mathfrak{L}^{4},\mathfrak{g})$. Then the \emph{energy} and \emph{frequency} of $\gamma$ observed by the static observers ${N}^{-1}\partial_{t}$,
\begin{align}\label{energy}
E&:=\mathfrak{g}(\dot{\gamma},N^{-1}\partial_{t})\quad\text{and}\quad \nu:=E/\hbar,
\end{align}
are constant along $\gamma$ if and only if $N\circ\gamma\equiv N_{0}$ for some $N_{0}\in\R$.
\end{lemma}
\begin{proof}
Using the warped structure of the spacetime \eqref{static} to decompose the geodesic $\gamma=(t,x)$, the geodesic equation $\ddot{\gamma}=0$ implies
\begin{align}\label{geodesic}
0=\left(\ddot{\gamma}\right)^{t}&=\ddot{t}+\frac{2\;\dot{(N\circ\gamma)}\;\dot{t}}{N\circ\gamma}.
\end{align}
This can be explicitly solved to say $\dot{t}=C\,(N\circ\gamma)^{-2}$
for some constant $C\in\R$. In consequence, \eqref{energy} simplifies to $\hbar\nu=E=-C\,(N\circ\gamma)^{-1}$ which is constant along $\gamma$ if and only if $N\circ\gamma\equiv N_{0}$ for some $N_{0}\in\R$.
\end{proof}

\subsection{Notation and conventions}
Our sign convention is such that the Ricci tensor $\Ric$ is constructed from the Riemannian curvature endomorphism $\Rm$ via
\begin{align}
{\Ric}_{ij}&={\Rm_{kij}}^{k}.
\end{align}
The second fundamental form $\rom{2}$ of an isometric embedding $(A^{n},a)\hookrightarrow (B^{n+1},b)$ of semi-Riemannian manifolds with corresponding unit normal vector field $\eta$ reads
\begin{align}
\rom{2}(X,Y)&:=b(\my{b}{\nabla}_{X} \eta,Y)
\end{align}
for all $X,Y\in\Gamma(A^{n})$, irrespective of the sign $\tau:=b(\eta,\eta)$. We will make use of the contracted Gau{\ss} equation
\begin{align}\label{Gauss}
\my{b}{\Scal}-2\tau\,\my{b}{\Ric}(\eta,\eta)&= \my{a}{\Scal}-\tau(\my{a}{\tr}\,\rom{2})^2+\tau\,\lvert {\rom{2}}\rvert^2,
\end{align}
where the left upper indices indicate the metric from which a certain covariant derivative or curvature tensor is constructed. Also, we will use the Codazzi equation
\begin{align}\label{Codazzi}
b(\,\my{b}{\Rm}(X,Y,\eta),Z)&=\left(\my{a}{\nabla}\!_{X}\rom{2}\right)(Y,Z)-\left(\my{a}{\nabla}\!_{Y}\rom{2}\right)(X,Z)
\end{align}
for all $X,Y,Z\in\Gamma(A^{n})$. Moreover, if $\tau=1$, we have
\begin{align}\label{eq:f}
\mylap{b}f=\mylap{a}f+\my{b}{\nabla}^2f(\eta,\eta)+(\my{a}{\tr}\,\rom{2})\,\eta(f)
\end{align}
for every smooth $f:B^{n+1}\to\R$. On $3$-dimensional manifolds $(A^{3},a)$, we will exploit the fact that the Weyl tensor vanishes so that the Riemann endomorphism can be algebraically reconstructed from the Ricci tensor and the metric via the Kulkarni-Nomizu product
\begin{align}\label{Kulkarni}
{\my{a}{\Rm}_{ijk}}^{l}&={\my{a}{\Ric}_{i}}^{l}a_{jk}-\my{a}{\Ric}_{ik}a_{j}^{l}-{\my{a}{\Ric}_{j}}^{l}a_{ik}+\!\my{a}{\Ric}_{jk}a_{i}^{l}-\frac{\!\!\my{a}{\Scal}}{2}\left(\delta_{i}^{l}a_{jk}-a_{ik}\delta_{j}^{l}\right).
\end{align}

In the proof of Theorem \ref{thm:main}, we will use the following notation for objects defined within a given AF-geometrostatic spacetime $(\mathfrak{L}^{4},\mathfrak{g})$: The $2$-dimensional intersection of the photon sphere $\photo$ and the time slice $\slice$ is called $\surf$. The level sets of the lapse function $N$ within the time slice $\slice$ will be called $\surf_{N}$, so that $\surf=:\surf_{N_{0}}$ for some $N_{0}\in\R^{+}$ as, by definition, the photon sphere is a level set of $N$. Recall that the surfaces $\surf_{N}$ and hence also $\surf$ must be of spherical topology as $N$ is assumed to regularly foliate the spacetime and thus also the embedded submanifold $\slice$ and because of Lemma \ref{lem:foli} -- at least if the mass $m$ of the spacetime is non-zero.

Tensor fields naturally living on the spacetime $(\mathfrak{L}^{4},\mathfrak{g})$ such as the Riemann curvature endomorphism $\mathfrak{Rm}$, the Ricci curvature $\mathfrak{Ric}$, the scalar curvature $\mathfrak{R}$ etc.\ will be denoted in gothic print. The metric induced on the photon sphere $\photo$ will be denoted by $p$, the induced metric on $\surf$ by $\sigma$, see Table \ref{table} on p.\ \pageref{table}.

We will also need to handle several second fundamental forms and unit normal vector fields. The second fundamental form of $(\photo,p)\hookrightarrow(\mathfrak{L}^{4},\mathfrak{g})$ will be called $\mathfrak{h}$, the mean curvature $\mathfrak{H}$, and the corresponding outward unit normal will be called $\nu$. The second fundamental form of $(\slice,g)\hookrightarrow(\mathfrak{L}^{4},\mathfrak{g})$ vanishes as the spacetime is static and the slice is 'canonical' and thus time-symmetric. The corresponding future pointing unit normal field is $N^{-1}\partial_{t}$. Similarly, the second fundamental form of $(\surf,\sigma)\hookrightarrow(\photo,p)$ vanishes, the future pointing unit normal is again $N^{-1}\partial_{t}$. Finally, the second fundamental form of $(\surf,\sigma)\hookrightarrow(\slice,g)$ will be denoted by $h$, the mean curvature by $H$, the outward unit normal coincides with $\nu$. The same notation will be used for $(\surf_{N},\sigma)$, see Table \ref{tableh} on p.\ \pageref{tableh}. The trace-free part of a symmetric $(0,2)$-tensor $T$ will be denoted by $\free{T}$.

\section{Proof of the main theorem}\label{sec:proof}
This section is dedicated to the proof of the following 'static photon sphere uniqueness theorem':
\begin{theorem}\label{thm:main}
Let $(\mathfrak{L}^4,\mathfrak{g})$ be an AF-geometrostatic spacetime possessing a photon sphere $\photo\hookrightarrow\mathfrak{L}^4$  with mean curvature $\mathfrak{H}$, arising as the inner boundary of $\mathfrak{L}^{4}$. Assume that the lapse function $N$ regularly foliates $\mathfrak{L}^4$. Then $\mathfrak{H}\equiv\text{const}$ and $(\mathfrak{L}^4,\mathfrak{g})$ is isometric to the \schild spacetime of the same  mass $m=1/(\sqrt{3}\,\mathfrak{H})>0$.
\end{theorem}

We will rely on the following proposition which is well-known in the literature, cf.\ e.\,g.\ \cite{CVE} (Theorem II.1) or \cite{Perlick} (Proposition 1).
\begin{Prop}\label{prop:umbilic}
Let $(\mathfrak{L}^{4},\mathfrak{g})$ be an AF-geometrostatic spacetime and $\photo\hookrightarrow\mathfrak{L}^{4}$ an embedded timelike hypersurface. Then $\photo$ is a photon surface if and only if it is totally umbilic, i.\,e.\ iff its second fundamental form is pure trace.
\end{Prop}

The following proposition asserts that photon spheres in AF-geometrostatic spacetimes have constant mean and constant scalar curvature. This is a special case of a more general fact about semi-Riemannian Einstein manifolds.
\begin{Prop}\label{CMC}
Let $n\geq2$ and let $(\mathfrak{L}^{n+1},\mathfrak{g})$ be a smooth semi-Riemannian manifold possessing an embedded totally umbilic hypersurface $(P^{n},p)\hookrightarrow(\mathfrak{L}^{n+1},\mathfrak{g})$, so that the second fundamental form $\mathfrak{h}$ is pure trace and thus satisfies
\begin{align}\label{umbilic}
\mathfrak{h}&=\frac{\mathfrak{H}}{n} p,
\end{align}
where $\mathfrak{H}$ denotes the mean curvature of $P^{n}$. If the semi-Riemannian manifold $(\mathfrak{L}^{n+1},\mathfrak{g})$ is Einstein or in other words if $\mathfrak{Ric}=\Lambda \mathfrak{g}$ for some constant $\Lambda\in\R$ then $P^{n}$ has constant mean curvature and constant scalar curvature
\begin{align}\label{eq:scalar}
\my{p}{\Scal}&\equiv(n+1-2\tau)\Lambda+\tau\,\frac{n-1}{n}\mathfrak{H}^{2},
\end{align}
where $\tau:=\mathfrak{g}(\eta,\eta)$. In particular, photon surfaces (and thus photon spheres) in AF-geometrostatic spacetimes are CMC and have constant scalar curvature
\begin{align}\label{eq:constant scalar 3}
\my{p}{\Scal}&\equiv\frac{2}{3}\,\mathfrak{H}^{2}.
\end{align}
\end{Prop}
\begin{proof}
Using \eqref{umbilic} and denoting the unit normal to $P^{n}$ (corresponding to $\mathfrak{h}$) by $\eta$, the Codazzi equation \eqref{Codazzi} reduces to
\begin{align}
\mathfrak{g}(\mathfrak{Rm}(X,Y,\eta),Z)&=\left(\my{p}{\nabla}\!_{X}\mathfrak{h}\right)(Y,Z)-\left(\my{p}{\nabla}\!_{Y}\mathfrak{h}\right)(X,Z)\\\nonumber
&=X(\mathfrak{H}/n)\,p(Y,Z)-Y(\mathfrak{H}/n)\,p(X,Z)
\end{align}
for all $X,Y,Z\in \Gamma(TP^{n})$. Contracting the $X$ and $Z$ slots and exploiting the antisymmetry of $\mathfrak{Rm}$, namely that $\mathfrak{g}(\mathfrak{Rm}(\eta,Y,\eta),\eta)=0$, we obtain
\begin{align}\label{1-n}
\mathfrak{Ric}(Y,\eta)&=Y(\mathfrak{H}/n)-nY(\mathfrak{H}/n)=(1-n)Y(\mathfrak{H}/n)
\end{align}
for all $Y\in\Gamma(TP^{n})$. The left hand side of \eqref{1-n} vanishes as $\mathfrak{g}$ is Einstein and $\mathfrak{g}(Y,\eta)=0$ which proves that $P^{n}$ is CMC as $Y\in\Gamma(TP^{n})$ was arbitrary. Furthermore, by the Gau{\ss} equation \eqref{Gauss} and \eqref{umbilic}, we find that
\begin{align}\label{eq:constant scalar}
\mathfrak{R}-2\tau\,\mathfrak{Ric}(\eta,\eta)&= \my{p}{\Scal}-\tau\,\mathfrak{H}^2+\tau\,\lvert {\mathfrak{h}}\rvert^2\\\nonumber
&=\my{p}{\Scal}-\tau\,\mathfrak{H}^{2}+\tau\,(\mathfrak{H}^{2}/n)\\\nonumber
&=\my{p}{\Scal}-\tau\,\frac{n-1}{n}\mathfrak{H}^{2}.
\end{align}
As $\mathfrak{g}$ is Einstein, we have $\mathfrak{Ric}(\eta,\eta)=\tau\,\Lambda$ and $\mathfrak{R}=(n+1)\Lambda$. Equation \eqref{eq:constant scalar} thus simplifies to \eqref{eq:scalar} so that $(P^{n},p)$ has constant scalar curvature as claimed.

Finally, AF-geometrostatic spacetimes are clearly Einstein with $\Lambda=0$ by the Einstein vacuum equation \eqref{EE}. Thus, by Proposition \ref{prop:umbilic}, photon surfaces (and hence photon spheres) in AF-geometrostatic spacetimes are CMC and have constant scalar curvature as in \eqref{eq:constant scalar 3}.
\end{proof}

Let us now proceed to prove Theorem \ref{thm:main}.
\begin{proof}[Proof of Theorem \ref{thm:main}]
Let $(\mathfrak{L}^4,\mathfrak{g})$ be an AF-geometrostatic spacetime as in the statement of the theorem, and let $\photo=\R\times\surf=:\R\times\surf_{N_{0}}$ be the photon sphere arising as the inner boundary of $\mathfrak{L}^{4}$. Let us first of all compute the second fundamental form $h$ of $(\surf,\sigma)\to(\slice,g)$. For $X,Y\in\Gamma(T\surf)$, we find
\begin{align*}
h(X,Y)&=g(\my{g}{\nabla}_{X}\nu,Y)=\mathfrak{g}(\my{\mathfrak{g}}{\nabla}_{X}\nu,Y)=\mathfrak{h}(X,Y)=\mathfrak{H}\,p(X,Y)/3=\mathfrak{H}\,\sigma(X,Y)/3,
\end{align*}
where we have used that $(\slice,g)$ is time-symmetric and that $(\photo,p)$ is totally umbilic in $(\mathfrak{L}^4,\mathfrak{g})$ by Proposition \ref{prop:umbilic}. As $\mathfrak{H}$ is constant by Proposition \ref{CMC}, this implies
\begin{align}\label{eq:umbilic CMC}
h&=\frac{\mathfrak{H}}{3}\sigma\quad\text{and thus}\quad H\equiv\frac{2}{3}\,\mathfrak{H},
\end{align}
so that the embedding $(\surf,\sigma)\hookrightarrow(M^{3},g)$ is totally umbilic and CMC. We will from now on write $H_{0}:=H=2\mathfrak{H}/3$. Using this information in the Codazzi-equation \eqref{Codazzi}, we get $g(\my{g}{\Rm}(X,Y,\nu),Z)=0$ for all $X,Y,Z\in\Gamma(T\surf)$ and thus 
\begin{align}\label{eq:Ric1}
\my{g}{\Ric}(X,\nu)&=0
\end{align}
for all $X\in\Gamma(T\surf)$ by contracting the $X$ and $Z$ slots and using the symmetry of the Riemann tensor. From \eqref{eq:Ric1} and the static metric equation \eqref{SMEvac1}, we deduce
\begin{align}\label{eq:nuN}
X(\nu(N))&=X(\nu(N))-\left(\my{g}{\nabla}_{X}\nu\right)(N)=\my{g}{\nabla^{2}}N(X,\nu)=N\,\my{g}{\Ric}(X,\nu)=0
\end{align}
for all $X\in\Gamma(T\surf)$ as $N$ is constant along $\surf$ by definition of photon spheres. This shows that $\nu(N)\equiv:\left[\nu(N)\right]_{0}$ is constant along $\surf$. From \eqref{SMEvac3} and \eqref{eq:nuN}, it can be seen that the mass parameter $m$ from Theorem \ref{thm:KM} (or in other words the ADM-mass of $(\slice,g)$) satisfies
\begin{align}\label{eq:mass}
m&=\frac{1}{4\pi}\int_{\surf}\nu(N)\,d\mu=\frac{\vert\surf\vert_{\sigma}}{4\pi}\left[\nu(N)\right]_{0},
\end{align}
where $\mu$ denotes the area measure with respect to $\sigma$, see also Section 4.2 in \cite{CDiss}.
 \vspace{1em}
\subsubsection*{Why the mass $m$ and $\nu(N)$ are non-zero}\label{comment}
Observe that $\vert\nu(N)\vert=\vert dN\vert_{g}$ on every level set $\surf_{N}$ of $N$ in $\slice$. Thus $\nu(N)\neq0$ on $\surf_{N}$ follows from the assumption that $N$ regularly foliates $\slice$. This\footnote{Alternatively, $m=0$ implies that $g$ is flat by \eqref{SMEvac1}, \eqref{eq:mass}, and \eqref{Kulkarni}, so that the spacetime is some exterior region of the Minkowski spacetime. The photon surfaces of the Minkowski spacetime are well-understood, see e.\,g.\ \cite{ONeill1}; in particular, the Minkowski spacetime does not possess a photon sphere in our sense.} ensures $m\neq0$ by \eqref{eq:mass}. 

By the maximum principle for elliptic PDEs (see e.\,g.\ \cite{GT}), by \eqref{SMEvac2}, and by the asymptotic condition that $N\to1$ as $r\to\infty$ required in the definition of AF-geometrostatic systems, Definition \ref{def:AFSVS}, $N$ will have values in the interval $I:=\left[N_0,1\right)$ or in the interval $I:=\left(1,N_{0}\right]$, where $N_{0}<1$ corresponds to positive and $N_{0}>1$ corresponds to negative mass $m$, see Lemma \ref{lem:foli}.

It will be convenient to use the \emph{area radius} of $\surf_{N}$ and $\surf=\surf_{N_{0}}$, defined by
\begin{align}\label{rN}
r(N)&:=\sqrt{\vert\surf_{N}\vert_{\sigma}/4\pi}\quad\text{and}\quad r_{0}:=r(N_{0}).
\end{align}

Applying \eqref{eq:f} to $f=N$ on $(\slice,g)$ and using \eqref{eq:nuN}, \eqref{SMEvac1}, \eqref{SMEvac3}, and \eqref{eq:umbilic CMC}, we find that\begin{align}\label{eq:Ric2}
N_{0}\,\my{g}{\Ric}(\nu,\nu)&\equiv-H_{0}\left[\nu(N)\right]_{0}.
\end{align}
When plugging this into the Gau{\ss} equation \eqref{Gauss} and remembering \eqref{SMEvac2} and \eqref{eq:umbilic CMC}, one gets
\begin{align}\label{eq:scal}
N_{0}\,\my{\sigma}{\Scal}&\equiv-2N_{0}\,\my{g}{\Ric}(\nu,\nu)+N_{0}H_{0}^2/2\equiv2 H_{0}\left[\nu(N)\right]_{0}+N_{0}H_{0}^2/2.
\end{align}
The Gau{\ss}-Bonnet theorem allows us to integrate \eqref{eq:scal} so that
\begin{align}\label{Gauss constraint}
4 N_{0}&=4mH_{0}+ r_{0}^{2}N_{0}H_{0}^{2}
\end{align}
by \eqref{eq:mass} and \eqref{rN}. The Gau{\ss} equation \eqref{Gauss} for $(\surf,\sigma)\hookrightarrow(\photo,p)$ implies that
\begin{align}\label{hihi}
\my{p}{\Scal}+2\,\my{p}{\Ric}(\frac{1}{N}\partial_{t},\frac{1}{N}\partial_{t})&= \my{\sigma}{\Scal},
\end{align}
on $\surf_{N_{0}}$. 
We know from Proposition \ref{CMC} that $\my{p}{\Scal}=2\mathfrak{H}^{2}/3$. The structure of the metric $p$ implies
\begin{align}
\my{p}{\Ric}(\frac{1}{N}\partial_{t},\frac{1}{N}\partial_{t})&=\frac{\mylap{\sigma}{N}}{N}=0
\end{align}
on $\surf=\surf_{N_{0}}$. Thus, \eqref{hihi} allows to compute
\begin{align}
\my{\sigma}{\Scal}&=\my{p}{\Scal}=2\mathfrak{H}^{2}/3.
\end{align}
The Gau{\ss}-Bonnet theorem leads to the explicit expression
\begin{align}\label{explicitr}
\mathfrak{H}\,r_{0}&=\pm\sqrt{3}
\end{align}
so that in particular $\mathfrak{H}\neq0$. Finally, from \eqref{Gauss constraint}, \eqref{explicitr}, and \eqref{eq:umbilic CMC}, we find
\begin{align}\label{explicitN}
0<N_{0}&=m\,\mathfrak{H}.
\end{align}
\subsubsection*{Handling the sign of $m$ and $\nu(N)$.} Other than it is done in Israel's analysis, we explicitly include the case of negative\footnote{Observe that $m=0$ has been ruled out above.} mass $m$ or in other words a priori allow $\nu(N)<0$ and $H<0$ along the photon sphere (by \eqref{eq:mass}, \eqref{explicitN}, and \eqref{eq:umbilic CMC}). In fact, this possibility can be ruled out by known results on the existence of outer trapped surfaces in static spacetimes, see Galloway \cite{Greg}. This implies that no smooth closed surface of constant negative mean curvature can be embedded into an AF-geometrostatic system as its inner boundary. However, we will not appeal to those arguments for the sake of demonstrating that our Israel style approach is flexible enough to directly handle negative mass/constant mean curvature of the photon sphere.

\subsubsection*{Rewriting the metric $g$ in adapted coordinates}
The next step imitates Israel's argument for static black hole uniqueness \cite{Israel} (as exposed in Heusler \cite{Heusler}). Because $\vert\nu(N)\vert=\vert dN\vert_{g}\neq0$ on $\slice$, the function $\rho:\slice\to\R^{+}$ given by
\begin{align}\label{rho}
\rho(x)&:=({\vert}\left.\nu(N)\right|_x{\vert})^{-1}\quad\text{for all }x\in\slice
\end{align}
is well-defined. As $N$ regularly foliates $\slice$, we can extend any coordinate system $(y^I),\,I=1,2$ on $U\subset\surf_{N_{0}}$ to the cylinder $I\times U$ by letting it flow along the (nowhere vanishing) gradient of $N$.  By construction, the metric $g$ reads 
\begin{align}\label{rhometric}
g=\rho^2\,dN^2+\sigma,
\end{align}
where $\sigma$ is the $2$-metric induced on $\surf_{N}$ (and depends on $N$!). { As $\nu(N)\neq0$, we can define a global sign
\begin{align}\label{def:lambda}
\lambda&:\equiv\signum(\nu(N))=\signum(m)=\signum(\mathfrak{H})=\signum(H_0)
\end{align}
by \eqref{eq:mass}, \eqref{explicitN}, and \eqref{eq:umbilic CMC}.}

In these variables, the static metric equations \eqref{SMEvac1}, \eqref{SMEvac3} imply the following identities
\begin{align}\label{eq:A}
0&=\frac{{\lambda}}{\rho}\left(\frac{H}{N}-H\!,_N-\frac{{\lambda}\rho}{2}\,H^2\right)-\frac{2}{\sqrt{\rho}}\,\mylap{2}\sqrt{\rho}-\frac{1}{2}\left[ \frac{\lvert \my{\sigma}{\grad}\rho\rvert_{\sigma}^2}{\rho^2}+2\lvert \free{h}\rvert^2_{\sigma}\right]\\\label{eq:B}
0&=\frac{{\lambda}}{\rho}\left(3\frac{H}{N}-H\!,_N\right)-\my{\sigma}{\Scal}-\mylap{2}\ln{\rho}-\left[ \frac{\lvert \my{\sigma}{\grad}\rho\rvert^2_{\sigma}}{\rho^2}+2\lvert \free{h}\rvert^2_{\sigma}\right]\\\label{eq:C}
0&=\rho,_N-{\lambda\,}\rho^2 H
\end{align}
 on any level set of $N$. Let $\mathfrak{s}:=\det(\sigma_{IJ})$. By definition of the second fundamental form $h$, we have $(\sqrt{\mathfrak{s}}),_N={\lambda}\sqrt{\mathfrak{s}}\,H\rho$. Using \eqref{eq:C} and non-negativity of the terms in square brackets, we obtain the following inequalities from \eqref{eq:A} and \eqref{eq:B}:
\begin{align}\label{ineq:A}
\partial_N\left(\frac{{\lambda}\sqrt{\mathfrak{s}}\,H}{\sqrt{\rho}\,N}\right)&\leq-2\frac{\sqrt{\mathfrak{s}}}{N}\,\,\mylap{2}\sqrt{\rho},\\\label{ineq:B}
\partial_N\left(\frac{\sqrt{\mathfrak{s}}}{\rho}\,\left[H N+\frac{4{\lambda}}{\rho}\right]\right)&\leq-N\sqrt{\mathfrak{s}}\left(\mylap{2}\ln{\rho}+\my{\sigma}{\Scal}\right),
\end{align}
holding on $\surf_{N}$. In these inequalities, equality holds if and only if the square brackets in \eqref{eq:A} and \eqref{eq:B} vanish i.\,e.\ iff $\nu(N)\equiv\text{const}$ and $\free{h}=0$ on the given level set. Integrating \eqref{ineq:A} from $N_{0}$ to $1$ and subsequently over $\surf$ (using a partition of unity corresponding to coordinate patches $U\subset\surf$), we get
\begin{align}\label{intA}
{\lambda}\left[\frac{1}{N}\int_{\surf_N}\frac{H}{\sqrt{\rho}}\,\,d\mu_N\right]_{N_0}^{1}&\leq-2 \int_{N_0}^{1}\frac{1}{N}\int_{\surf_N}\mylap{2}\sqrt{\rho}\,\,d\mu_N\,dN=0
\end{align}
 from Fubini's theorem, where $\mu_{N}$ denotes the area measure w.\,r.\,t.\ $\sigma$ on $\surf_{N}$. The right-hand side of \eqref{intA} vanishes by the divergence theorem. Now $H,\rho\equiv\text{const}$ on $\surf_{N_{0}}$ by \eqref{eq:umbilic CMC} and \eqref{eq:nuN}. Moreover, Theorem \ref{thm:KM} and Lemma \ref{lem:foli} allow us to compute that $H=\frac{2}{r}+\mathcal{O}(r^{-2})$ and ${\rho(N)}=\frac{r^{2}}{{\vert}m{\vert}}+\mathcal{O}(r)$ asymptotically as $r\to\infty$. Combining this with \eqref{intA} and the definition of the area radius \eqref{rN}, we find
\begin{align}
\frac{4\pi r_{0}^{2} H_{0}\,\sqrt{{\lambda}\left[\nu(N)\right]_{0}}}{N_0}{\lambda}&\geq\lim_{r\to\infty}\frac{{\lambda}}{N}\int_{\surf_N}H\sqrt{{\lambda\,}\nu(N)}\,d\mu_N= 8\pi\sqrt{{\vert}m{\vert}}{\lambda},
\end{align}
where $H_{0}$ denotes the mean curvature of $\surf_{N_{0}}$. Using \eqref{eq:mass}, this simplifies to
{\begin{align}\label{intineqA}
\lambda\left(2N_{0}-r_{0}H_{0}\right)&\leq0.
\end{align}}

Furthermore, by Fubini's theorem, the divergence theorem, and the Gau{\ss}-Bonnet theorem, integrating inequality \eqref{ineq:B} from $N_{0}$ to $1$ and subsequently over $\surf$ (with a partition of unity as before) gives
\begin{align}
\left[\;\int_{\surf_N}\frac{1}{\rho}\left[HN+\frac{4{\lambda}}{\rho}\right]d\mu_{N}\right]_{N_0}^1&\leq-\int_{N_0}^{1}N\int_{\surf_N}\left(\mylap{2}\ln{\rho}+\my{\sigma}{\Scal}\right)\,d\mu_N\,dN\\
&=-8\pi\,\int_{N_0}^{1}N\,dN=-4\pi(1-N_0^2).
\end{align}
Again making use of the discussed asymptotics, \eqref{eq:nuN}, and \eqref{eq:umbilic CMC}, we obtain
\begin{align}
{\lambda}\left[\nu(N)\right]_{0}\left[H_{0}N_0+4\left[\nu(N)\right]_{0}\right]\lvert\surf_{N_{0}}\rvert&\geq 4\pi(1-N_0^2).
\end{align}
By \eqref{eq:mass} and \eqref{rN}, this simplifies to
\begin{align}\label{intineqB}
{\vert}m{\vert}\left[H_{0}N_{0}+\frac{4m}{r_{0}^{2}}\right]&\geq 1-N_0^2.
\end{align}
{ Using \eqref{eq:umbilic CMC}, \eqref{explicitr}, and \eqref{explicitN}, we find
\begin{align*}
 r_0^2\leq (6\lambda+3)m^2
\end{align*}
which rules out $\lambda=-1$. Thus, $m>0$ and $\mathfrak{H}>0$ by \eqref{def:lambda}.}
We now estimate on the one hand that
\begin{align}\label{eq:1}
1&\stackrel{\eqref{Gauss constraint}}{=}\frac{1}{4N_{0}}\left(4mH_{0}+ r_{0}^{2}N_{0}H_{0}^{2}\right)\stackrel{\eqref{intineqA}}{\geq}\frac{2m}{r_{0}}+N_{0}^{2}\quad\Leftrightarrow\quad N_{0}^{2}\leq 1-\frac{2m}{r_{0}}
\end{align}
and on the other hand that
\begin{align}\label{eq:2}
2&\stackrel{\eqref{Gauss constraint}}{=}\frac{1}{2N_{0}}\left[4mH_{0}+ r_{0}^{2}N_{0}H_{0}^{2}\right]=\frac{H_{0}r_{0}^{2}}{2N_{0}}\left[H_{0}N_{0}+\frac{4m}{r_{0}^{2}}\right]\\\nonumber
&\stackrel{\eqref{intineqB}}{\geq}\frac{H_{0}r_{0}^{2}(1-N_{0}^{2})}{2mN_{0}}\stackrel{\eqref{intineqA}}{\geq}\frac{(1-N_{0}^{2})r_0}{m}\\\nonumber
&\!\!\!\!\Leftrightarrow\quad N_{0}^{2}\geq 1-\frac{2m}{r_{0}}.
\end{align}
Combining \eqref{eq:1} and \eqref{eq:2} gives $N_{0}=\sqrt{1-2m/r_{0}}$ just as in Schwarzschild. Both inequalities \eqref{eq:1} and \eqref{eq:2} are sharp so that $H_{0}=2N_{0}/r_{0}$. This, together with \eqref{explicitr}, \eqref{explicitN}, and \eqref{eq:umbilic CMC} gives $m=1/(\sqrt{3}\,\mathfrak{H})>0$ so that the parameter $\mathfrak{H}>0$ determines the (positive) mass of the spacetime.

As discussed above, this also implies equality in both \eqref{ineq:A} and \eqref{ineq:B} which gives us $\free{h}\equiv0$ and $\rho\equiv\text{const}$ on every $\surf_N$. By \eqref{eq:C}, we find that $H$ must also be constant on every $\surf_{N}$. \eqref{eq:B} and \eqref{eq:A} then imply that the Gau{\ss} curvature must be constant on every level $\surf_N$. This, in turn, tells us that $\sigma=r^{2}\Omega$ on $\surf_{N}$ with $r=r(N)$ by the uniformization theorem, with $\Omega$ the canonical metric on $\mathbb{S}^2$ as above. From \eqref{eq:mass}, we know that $\rho=r^2/m$. Using this and the area radius $r(N)$ defined in \eqref{rN}, \eqref{eq:C} implies $\frac{dr}{dN}>0$ if $H>0$ on $\surf_{N}$. However, \eqref{eq:A} can be integrated explicitly to say 
\begin{align}
H&=A\,N\,\exp\left(-\int_{N_{0}}^{N}(\rho H)\,dN\right)
\end{align}
on $\surf_{N}$ for some $A\in\R$. As $H_{0}>0$, also $A>0$ and thus $H>0$ on all $\surf_{N}$. Thus $r(N)$ is invertible, we denote its inverse function by $N(r)$.

At this point, we know by chain rule that 
\begin{align}
g&=\rho(N(r))^2\,(N'(r))^{2}\,dr^2+r^2\Omega,
\end{align}
where $N'$ is the $r$-derivative of $N$. This shows that the spacetime is spherically symmetric so that the claim follows from a direct computation or from Birkhoff's theorem, see e.\,g.\ \cite{Wald}. For the sake of completeness, we will demonstrate the direct computation, here:

Equation \eqref{eq:C} tells us that
$H=2m/(r^{3}N')$ by chain rule, where $N'$ must be non-zero as $N$ is invertible. This and \eqref{eq:A} combine to an ODE for $N$, namely
\begin{align}
N\,N''&=-\frac{2N}{r}N'-(N')^{2}\quad\Leftrightarrow\quad u''=-2u'/r,
\end{align}
using $u:=N^{2}$. This ODE can be solved explicitly to read $N(r)=\sqrt{A+B/r}$ with $A,B\in\R$. From the asymptotic convergence $N\to1$ as $r\to\infty$, we deduce $A=1$. The explicit value $N_{0}=\sqrt{1-2m/r_{0}}$ allows us to compute $B=-2m$ so that 
\begin{align}
N(r)&=\sqrt{1-2m/r}\quad\text{and thus}\\
g&=\rho(N(r))^2\,(N'(r))^{2}\,dr^2+r^2\Omega=\frac{1}{N^{2}}\,dr^{2}+r^{2}\Omega.
\end{align}
This proves that the spacetime $(\mathfrak{L}^{4})$ is isometric to (an exterior region of) the \schild black hole spacetime $(\overline{\mathfrak{L}}^{4},\overline{\mathfrak{g}})$ from \eqref{schwarzmetric} of mass $m=\frac{1}{\sqrt{3}\,\mathfrak{H}}>0$.
\end{proof}
\vfill\vfill\vfill\vfill
\begin{figure}[h]
\begin{tabular}{c|c|c|c|c}
&&&&\\[-.5em]
name & manifold & metric & tensors/operators & indices\\[.5em]\hline\hline
&&&&\\[-.25em]
spacetime & $\mathfrak{L}^{4}=\R\times\slice$ & $\mathfrak{g}$ & $\mathfrak{Ric}$, $\my{\mathfrak{g}}{\nabla}$, $\dots$ & $\alpha,\beta,\dots$\\[.5em]\hline
&&&&\\[-.25em]
photon sphere & $\photo=\R\times\surf$ & $p$ & $\my{p}{\Ric}$, $\my{p}{\nabla}$, $\dots$ & $a,b,\dots$\\[.5em]\hline
&&&&\\[-.25em]
time slice & $\slice$ & $g$ & $\my{g}{\Ric}$, $\my{g}{\nabla}$, $\dots$ & $i,j,\dots$\\[.5em]\hline
&&&&\\[-.25em]
photon sphere (base) & $\surf=\surf_{N_{0}}$ & $\sigma$ & $\my{\sigma}{\Ric}$, $\my{\sigma}{\nabla}$, $\dots$ & $I,J,\dots$\\[.5em]\hline
&&&&\\[-.25em]
$N$-level in $\slice$ & $\surf_{N}$ & $\sigma$ & $\my{\sigma}{\Ric}$, $\my{\sigma}\nabla$, $\dots$ & $I,J,\dots$\\[.5em]
\end{tabular}
\caption{Notational conventions for manifolds, metrics, induced tensor fields, induced differential operators, and coordinate indices.\label{table}}
\end{figure}
\vfill
\phantom{hallo}
\begin{figure}[h]
\begin{tabular}{c|c|c|c}
&&&\\[-.5em]
embedding & second fund.\ form & mean curvature & normal vector\\[.5em]\hline\hline
&&&\\[-.25em]
$(\photo,p)\hookrightarrow(\mathfrak{L}^{4},\mathfrak{g})$ & $\mathfrak{h}$ & $\mathfrak{H}$ & $\nu$\\[.5em]\hline
&&&\\[-.25em]
$(\slice,g)\hookrightarrow(\mathfrak{L}^{4},\mathfrak{g})$ & $0$ & $0$ & $N^{-1}\partial_{t}$\\[.5em]\hline
&&&\\[-.25em]
$(\surf,\sigma)\hookrightarrow(\photo,p)$ & $0$ & $0$ & $N^{-1}\partial_{t}$\\[.5em]\hline
&&&\\[-.25em]
$(\surf,\sigma)\hookrightarrow(\slice,g)$ & $h$ & $H$ & $\nu$\\[.5em]\hline
&&&\\[-.25em]
$(\surf_{N},\sigma)\hookrightarrow(\slice,g)$ & $h$ & $H$ & $\nu$
\end{tabular}
\caption{Notational conventions for second fundamental form, mean curvature, and normal vectors.\label{tableh}}
\end{figure}
\vfill\vfill
\bibliographystyle{amsplain}
\bibliography{photon-sphere}

\providecommand{\bysame}{\leavevmode\hbox to3em{\hrulefill}\thinspace}
\providecommand{\MR}{\relax\ifhmode\unskip\space\fi MR }
\providecommand{\MRhref}[2]{%
  \href{http://www.ams.org/mathscinet-getitem?mr=#1}{#2}
}
\providecommand{\href}[2]{#2}
\begin{thebibliography}{10}

\bibitem{ADM}
Richard Arnowitt, Stanley Deser, and Charles~W. Misner, \emph{{Coordinate
  Invariance and Energy Expressions in General Relativity}}, Phys. Rev.
  \textbf{122} (1961), no.~3, 997--1006.

\bibitem{Bartnik}
Robert Bartnik, \emph{{The Mass of an Asymptotically Flat Manifold}},
  Communications on Pure and Applied Mathematics \textbf{39} (1986), 661--693.

\bibitem{BMuA}
Gary~L. Bunting and Abdul Kasem~Muhammad Masood-ul Alam, \emph{Nonexistence of
  multiple black holes in asymptotically euclidean static vacuum space-time},
  General Relativity and Gravitation \textbf{19} (1987), no.~2, 147--154
  (English).

\bibitem{CDiss}
Carla Cederbaum, \emph{The newtonian limit of geometrostatics}, Ph.D. thesis,
  FU Berlin, 2012,
  \href{http://arxiv.org/abs/1201.5433v1}{arXiv:\linebreak[1]1201.5433v1}.

\bibitem{cedergal}
Carla Cederbaum and Gregory Galloway, in preparation, 2014.

\bibitem{CVE}
Clarissa-Marie Claudel, Kumar~Shwetketu Virbhadra, and George F.~R. Ellis,
  \emph{{The Geometry of Photon Surfaces}}, J. Math. Phys. \textbf{42} (2001),
  no.~2, 818--839.

\bibitem{DR}
Mihalis Dafermos and Igor Rodnianski, \emph{Lectures on black holes and linear
  waves}, arXiv:0811.0354v1, 2008.

\bibitem{Frankel}
Theodore Frankel, \emph{{Gravitational Curvature: An Introduction to Einstein's
  Theory}}, Dover Books on Physics, 2011.

\bibitem{Greg}
Gregory Galloway, \emph{On the topology of black holes}, Commun. Math. Phys.
  \textbf{151} (1993), 53--66.

\bibitem{GT}
David Gilbarg and Neil~S. Trudinger, \emph{{Elliptic partial differential
  equations of second order}}, 2 ed., Springer, 1970.

\bibitem{Heusler}
Markus Heusler, \emph{{Black Hole Uniqueness Theorems}}, Cambridge Lecture
  Notes in Physics, Camb. Univ. Press, 1996.

\bibitem{Israel}
Werner Israel, \emph{{Event Horizons in Static Vacuum Space-Times}}, Phys. Rev.
  \textbf{164} (1967), no.~5, 1776--1779.

\bibitem{KM}
Daniel Kennefick and Niall~\'O Murchadha, \emph{{Weakly decaying asymptotically
  flat static and stationary solutions to the Einstein equations}}, Class.
  Quantum Grav. \textbf{12} (1995), no.~1, 149.

\bibitem{MiaoUnique}
Pengzi Miao, \emph{A remark on boundary effects in static vacuum initial data
  sets}, Class. Quantum Grav. \textbf{22} (2005), no.~11, L53--L59.

\bibitem{ONeill1}
Barrett O'Neill, \emph{{Semi-Riemannian Geometry With Applications to
  Relativity}}, {Academic Press}, 1983.

\bibitem{ONeill2}
\bysame, \emph{{The Geometry of Kerr Black Holes}}, {Dover Books on Physics},
  1992.

\bibitem{Perlick}
Volker Perlick, \emph{On totally umbilici submanifolds of semi-riemannian
  manifolds}, Nonlinear Analysis \textbf{63} (2005), no.~5-7, e511--e518.

\bibitem{VE1}
Kumar~Shwetketu Virbhadra and George F.~R. Ellis, \emph{{Schwarzschild black
  hole lensing}}, Phys. Rev. D \textbf{62} (2000), 084003.

\bibitem{VE2}
\bysame, \emph{{Gravitational lensing by naked singularities}}, Phys. Rev. D
  \textbf{65} (2002), 103004.

\bibitem{Wald}
Robert~M. Wald, \emph{{General Relativity}}, The University of Chicago Press,
  1984.

\bibitem{MzH}
Henning~M\"uller zum Hagen, \emph{{On the analyticity of static vacuum
  solutions of Einstein's equations}}, Proc. Camb. Phil. Soc. \textbf{67}
  (1970), 415--421.

\end{thebibliography}
\nocite{Frankel}
\end{document}